\documentclass[a4paper,12pt]{amsart}
\usepackage{amsmath}
\usepackage{amssymb}
\usepackage{mathrsfs}
\usepackage{color}
\usepackage{epsfig}
\usepackage{graphicx}
\usepackage[all]{xy}
\usepackage{enumerate}

\theoremstyle{plain}
\newtheorem{theorem}{Theorem}[section]
\newtheorem{proposition}[theorem]{Proposition}
\newtheorem{corollary}[theorem]{Corollary}
\newtheorem{lemma}[theorem]{Lemma}

\theoremstyle{definition}
\newtheorem{remark}[theorem]{Remark}

\newtheorem{definition}[theorem]{Definition}

\newtheorem{assumption}[theorem]{Assumption}

\newcommand{\C}{\mathbb{C}}
\newcommand{\R}{\mathbb{R}}
\newcommand{\Z}{\mathbb{Z}}

\newcommand{\g}{\mathfrak g}
\newcommand{\h}{\mathfrak h}
\newcommand{\SF}{\mathscr F}
\newcommand{\CO}{\mathcal O}

\renewcommand{\tilde}{\widetilde}
\renewcommand{\setminus}{\smallsetminus}
\newcommand{\nin}{/\kern-2.1ex\in}

\def\<{\left\langle}
\def\>{\right\rangle}
\def\End{\operatorname{End}}

\def\ind{\operatorname{ind}}

\numberwithin{equation}{section}

\title[Torus fibrations and localization of index III]
{Torus fibrations and localization of index III \\ 
-Equivariant version and its applications-}

\author[H. Fujita]{Hajime Fujita}
\author[M. Furuta]{Mikio Furuta}
 \author[T. Yoshida]{Takahiko Yoshida}

\subjclass[2000]{} 

\address{Department of Mathematical and Physical Sciences, Japan Womens's University, 
2-8-1 Mejirodai, Bunkyo-ku, Tokyo, 112-8681, Japan}
\email{fujitah@fc.jwu.ac.jp}

\address{Graduate School of Mathematical Sciences, The University of Tokyo, 
3-8-1 Komaba, Meguro-ku, Tokyo, 153-8914, Japan}
\email{furuta@ms.u-tokyo.ac.jp}

\address{Department of Mathematics, Graduate School of Science and Technology, Meiji University, 1-1-1 Higashimita, Tama-ku, Kawasaki, 214-8571, Japan}
\email{takahiko@meiji.ac.jp}

\begin{document}

\maketitle

\begin{abstract}
This paper is the third of the series concerning the localization of
the index of Dirac-type operators. In our previous papers we gave a
formulation of index of Dirac-type operators on  open manifolds under
some geometric setting, whose typical example was given by
the structure of a torus fiber bundle on the ends of the open manifolds. 
We introduce two equivariant versions of the localization.
As an application we give a proof of Guillemin-Sternberg's quantization 
conjecture in the case of torus action.
\end{abstract}


\section{Introduction}


This is the third of the series concerning a localization of the
index of elliptic operators.
The localization of integral is a mechanism by which
the integral of a differential form on a manifold 
becomes equal to the integral of another differential
form on a submanifold, which has been formulated
under various geometric settings.
The {\it submanifold} is either an open submanifold
or a closed submanifold. 
When it is an open submanifold,
the localization is closely related to some {\it excision 
formula}. When it is a closed submanifold,
the localization is usually obtained by applying 
some {\it product formula} to the normal bundle of
the submanifold after the localization 
to the open tubular neighborhood.


A typical geometric setting for such localization
is given by action of compact Lie group, and
a localization is formulated  in terms of 
the equivariant de Rham cohomology groups.
An example is Duistermaat and Heckman's formula on a symplectic manifold.
It is formally possible to replace the equivariant de Rham cohomology groups
with the equivariant $K$-cohomology groups, and the resulting localization
in terms of the equivariant $K$-cohomology groups is known as
Atiyah-Segal's Lefschetz formula of equivariant index.


In our previous papers \cite{Fujita-Furuta-Yoshida1,Fujita-Furuta-Yoshida2}
the geometric setting ensuring our localization is
typically given by the structure of a torus fiber bundle.
Under this setting we consider  the Riemann-Roch number  or the index of
the Dolbeault operator or the Dirac-type operator, 
associated to an almost complex structure or a spin{$^c$} structure,
which is twisted by some vector bundle.
We do not assume any global group action.
Instead, on the vector bundle, we assume 
a family of flat connections of the fibers of the torus bundle.
Our setting has generalized from the setting of a single torus bundle
structure to the case that we have a finite open covering and
a family of torus bundle structures on the open sets which
satisfy some compatibility condition.
The dimensions of the fibers of the family of torus bundle structures
can vary.
This generalization was necessary to formulate 
a product formula in a full form \cite{Fujita-Furuta-Yoshida2},
and the product formula is used to compute the local contribution in some examples.


In this paper we introduce the equivariant version of 
our localization.
When a compact Lie group $G$ acts on everything,
it is straightforward to generalize our previous argument.
The index takes values in the character ring $R(G)$
and we have the Riemann-Roch character.
We go further. Suppose two compact Lie group $G$ and $K$
acts on everything simultaneously and assume that
their actions are commutative.
In this paper we formulate another type of equivariant
version as follows.
The main assumption of our previous papers in our geometric setting
was the vanishing of the de Rham cohomology groups
with some local coefficients on each fiber of the torus bundles.
Our new setting is given by weaken the assumption.
Roughly speaking our new assumption is
that only the $G$-invariant part of the de Rham cohomology groups
vanish.
Under this new weaker assumption,
the full $G \times K$-equivariant index is not well defined.
Instead only the $G$-invariant part of the $G\times K$-equivariant index
is well defined as an element of the character ring of $K$.


As an application of the latter equivariant version
we give a proof of Guillemin-Sternberg's quantization conjecture
in the case of torus action.


Our localization is basically a purely topological statement.
It would be required to formulate it as the equality
between topological index and analytical index.
The definition of topological index is, however, not given at the present.


In this paper we work in the smooth category.
In Section~2 and Section~3 we describe the orbifold version of {\it compatible fibration } and {\it compatible system} in the previous papers \cite{Fujita-Furuta-Yoshida1,Fujita-Furuta-Yoshida2} and their equivariant versions. 
We give several definitions under the same names as in the previous papers,
though the notions are generalized as well as the propositions there.
We also introduce the notion of {\it $G$-acyclicity} of equivariant compatible system. 
In Section~4 we define the {\it local index} for acyclic compatible system, 
the {\it $K$-equivariant local index} for $K$-equivariant acyclic compatible system and {\it $G$-invariant local index} for $G$-acyclic compatible system. 
We also give 
one of our main theorem (Theorem~\ref{equivariant localization for invariant part}). 
We define the local index under the assumption that the compatible fibration is 
constructed from a torus action. 
Though the torus action is not essential to define the index, it is useful to avoid discussing technical conditions and enough to have an application to the quantization conjecture. 
In Section 5, we explain
the construction of the $K$-equivariant $G$-acyclic compatible system  using an action of a torus $G$ with a simultaneous action of a compact Lie group $K$
on an almost complex manifold.
In Section 6, as a preparation of the proof of quantization conjecture,
we show a vanishing property of the $G$-invariant local index when $G$ is the circle group $S^1$ under some condition.
In Section 7 we give a proof of quantization conjecture for torus action.

\section{Equivariant compatible fibration}

In this section we recall the notion of {\it compatible fibration} in \cite{Fujita-Furuta-Yoshida2}, and introduce its equivariant version. 
The definitions given in this section are generalized versions 
of those given in \cite{Fujita-Furuta-Yoshida2} so that 
we avoid dealing with the orbifold singularities directly which come from finite isotropies of a torus action.
\subsection{Compatible fibration}

Let $M$ be a smooth manifold. 
\begin{definition}\label{compatible fibration}
A {\it compatible fibration on $M$} is a collection of the data 
$\{V_{\alpha}, {\SF}_{\alpha}\}_{\alpha\in A}$ consisting of 
an open covering $\{V_{\alpha}\}_{\alpha\in A}$ of $M$ and 
a foliation $\SF_{\alpha}$ on $V_{\alpha}$ with compact leaves 
which  satisfies the following properties.
\begin{enumerate}
\item The holonomy group of each leaf of $\SF_\alpha$ is finite. 
\item\label{correspondence between foliation and pi_alpha}For each $\alpha$, $\beta \in A$, if a leaf $L\in \SF_\alpha$ has non-empty intersection $L\cap V_\beta\neq \emptyset$, then, $L\subset V_\beta$. 
\end{enumerate}
\end{definition}

Let $\{ V_\alpha ,\SF_\alpha\}_{\alpha \in A}$ be a compatible fibration on $M$. 

\begin{definition}\label{admissible subset}
A subset $C$ of $M$ is said to be {\it admissible} if, 
on each $V_\alpha\cap C\neq \emptyset$, $C$ contains all leaves $L\in \SF_\alpha$ which intersect with $C$. 
\end{definition}
For an admissible subset $C$ we define the foliation $\SF_\alpha|_C$ on $C\cap V_\alpha$ by
\[
\SF_\alpha|_C=\{ L\in \SF_\alpha \mid L\cap C\neq\emptyset \} .
\]
\begin{proposition}
Let $C$ be an admissible submanifold of $M$. Then, $\{ C\cap V_\alpha, \SF_\alpha|_C\}_\alpha$ is a compatible fibration on $C$. 
\end{proposition}

\begin{definition}
A function $f:M\to \R$ is said to be {\it admissible} if $f$ is constant along leaves of $\SF_\alpha$ for all $\alpha\in A$.
\end{definition}

\subsection{$K$-equivariant compatible fibration}

Let $K$ be a compact Lie group. We rigorously describe the assumption on a $K$-action since we deal with the orbifold setting. 
Suppose that there exists an action of $K$ on $M$ which preserves all these data. 
Let $\{ V_\alpha ,\SF_\alpha \}_{\alpha \in A}$ be a compatible fibration on $M$.

\begin{definition}\label{equivariant compatible fibration}
A compatible fibration $\{ V_\alpha ,\SF_\alpha \}_{\alpha \in A}$ is said to be {\it $K$-equivariant} if it satisfies the following conditions. 
\begin{enumerate}
\item The $K$-action preserves $V_\alpha$'s. 
\item The $K$-action preserves the foliation $\SF_\alpha$ on $V_\alpha$. We allow that the $K$-action sends a leaf to another leaf. 
\end{enumerate}
\end{definition}

\subsection{$G\times K$-equivariant compatible fibration for torus action}
\label{torus : compatible fibration}
In this subsection we construct a structure of equivariant compatible fibration
 using torus action. 

Let $G$ be a compact torus and $K$ a compact Lie group. 
Let $V$ be a smooth manifold equipped with an 
action of the product $G\times K$. 
We denote by $G_x$ the stabilizer subgroup of $G$ at a point $x\in V$. 
Let $A$ be the set of subgroups of $G$ which appears as the identity component 
of the stabilizer group at some point $x\in V$. 
We assume the following condition. 

\begin{assumption}
$A$ is a finite set. 
\end{assumption}
%
%


We endow $G$ with a {\it rational} flat Riemannian metric. 
Precisely speaking we take a Euclidean metric on the Lie algebra of $G$ such that the intersection of the integral lattice and the lattice generated by some orthonormal basis is a sublattice of maximal rank. 
We extend the metric to the whole $G$. 
Let $H$ be a closed subgroup of $G$.  
We denote by $H^{o}$ the identity component of $H$. 
Let $H^{\bot}$ be the orthogonal complement of $H^o$ 
defined as the image of the orthogonal complement of the Lie algebra of $H$ 
by the exponential map. 
Since the metric is rational $H^{\bot}$ is 
well-defined as a compact connected subgroup of $G$ and 
it has only finitely many intersection points $H \cap H^{\bot}$. 
Recall that there is a {\it good open covering} with respect to the $G$-action. 

\begin{lemma}[Existence of a good open 
covering, Lemma~5.3 in \cite{Fujita-Furuta-Yoshida2}]
\label{good open covering}
There exists an open covering $\{V_H \}_{H \in A}$ of $V$ 
parameterized by $A$  satisfying the following properties. 
\begin{enumerate}
\item Each $V_H$ is $G$-invariant.
\item 
For each $x \in V_H$ we have $G_x^o \subset H$.
\item If $V_H \cap V_{H'}\neq \emptyset$, then we have $H \subset H'$ or $H \supset H'$.
\item If for each $x\in V$ there is an open neighborhood $V_x$ of $x$ such that $V_{gx}=gV_x$ for all $g\in G$, then $\{V_H\}_{H\in A}$ satisfies $V_H\subset \bigcup_{G_x=H}V_x$. 
\end{enumerate}
\end{lemma}

\begin{proposition}\label{rem good open covering}
Let $\{V_H\}_{H\in A}$ be the good open covering as in 
Lemma~\ref{good open covering}.
There exists a structure of $G\times K$-equivariant compatible fibration on 
$\{V_H\}_{H\in A}$. 
\end{proposition}
\begin{proof}
Since each $V_H$ is 
constructed from an open neighborhood of the closed subset 
consisting of points whose stabilizer is equal to $H$, 
we may assume that $V_H$ is $G\times K$-invariant. 
Each $G$-invariant open subset 
$V_H$ has a structure of a foliation $\SF_{H}$ via 
the decomposition into the union of $H^{\perp}$-orbits.  
Each leaf of $\SF_{H}$ is a $H^{\perp}$-orbit through 
some point $x\in V_H$ and its holonomy group is equal to $G_x\cap H^{\perp}$. 
Note that $G_x\cap H^{\perp}$ is a finite group because 
$G_x$ is a subgroup of $H$ by the second property in Lemma~\ref{good open covering}.  
Then, $\{V_{H}, \SF_H\}_{H\in A}$ 
is a $G\times K$-equivariant compatible fibration on $V$. 
\end{proof}

\begin{remark}\label{S^1-case}
If $G$ is the circle group $S^1$, 
then the finite set $A$ consists of a single element. 
In this case we use the open covering $\{V_H\}_{H\in A}$ 
consisting of the single open set $V$.  
\end{remark}

\section{Equivariant compatible system}
In this section we recall the definition of 
{\it compatible system}, its {\it acyclicity} and their equivariant versions. 
We introduce the notion of {\it $G$-acyclicity}, which is a variant of usual acyclicity for the equivariant setting. 
As for the compatible system, 
definitions in this section are generalized versions 
of those given in \cite{Fujita-Furuta-Yoshida2}. 

Let $(V,g)$ be a Riemannian manifold, $W$ a $\Z /2$-graded $Cl(TV)$-module bundle on $V$. 
Suppose that $V$ is equipped with a compatible fibration $\{V_{\alpha}, {\SF}_{\alpha}\}_{\alpha\in A}$. 
In the rest of this paper we impose the following conditions on the Riemannian metric $g$. 
\begin{assumption}\label{assumption for Riemannian metric}
Let $\nu_\alpha=\{ u\in TV_\alpha \mid g(u,v)=0\ \text{for all }v\in T\SF_\alpha \}$ be the normal bundle of $\SF_\alpha$. Then, $g|_{\nu_\alpha}$ is invariant under holonomy, and gives a transverse invariant metric on $\nu_\alpha$.
\end{assumption}
\begin{remark}
The above Riemannian metric is an orbifold version of a compatible Riemannian metric which is actually used in \cite{Fujita-Furuta-Yoshida2}. 
If the compatible fibration is defined by using a torus action, then any torus invariant metric satisfies the assumption.  
\end{remark}
\subsection{Compatible system}

\begin{definition}\label{compatible system}
A {\it compatible system} on $(\{V_{\alpha}, \SF_{\alpha}\}, W)$ is a data $\{ D_{\alpha}\}_{\alpha \in A}$ satisfying the following properties. 
\begin{enumerate}
\item $D_{\alpha}\colon \Gamma (W|_{V_{\alpha}})\to \Gamma (W|_{V_{\alpha}})$ is a first order formally self-adjoint differential operator of degree-one.
\item $D_{\alpha}$ contains only the derivatives along leaves of $\SF_{\alpha}$. 
\item The principal symbol $\sigma (D_{\alpha})$ of $D_{\alpha}$ is given by $\sigma (D_{\alpha})=c\circ p_{\alpha}\circ \iota_{\alpha}^*\colon T^*V_{\alpha}\to \End (W|_{V_{\alpha}})$, where $\iota_{\alpha}\colon T\SF_{\alpha}\to TV_{\alpha}$ is the natural inclusion from the tangent bundle along leaves of $\SF_\alpha$ to $TV_\alpha$, $p_{\alpha}\colon T^*\SF_{\alpha}\to T\SF_{\alpha}$ is the isomorphism induced by the Riemannian metric and $c\colon T\SF_{\alpha}\to \End (W|_{V_{\alpha}})$ is the Clifford multiplication. 
\item For a leaf $L\in \SF_\alpha$ let $\tilde u\in \Gamma (\nu_\alpha|_L)$ be a section of $\nu_\alpha|_L$ parallel along $L$. 
$\tilde{u}$ acts on $W|_L$ by the Clifford multiplication $c(\tilde{u})$. Then $D_{\alpha}$ and $c(\tilde{u})$ anti-commute each other, i.e. 
\[
0=\{ D_{\alpha},c(\tilde{u}) \}:=
D_{\alpha}\circ c(\tilde{u})+c(\tilde{u})\circ D_{\alpha}
\]
as an operator on $W|_L$. 
\end{enumerate}
\end{definition}

The above definitions of the compatible fibration and 
the compatible system are introduced to avoid dealing with the orbifold 
singularities directly which come from finite isotropies of a torus action. 
The following lemma guarantees that we have an orbifold chart for each leaf, 
and analytic estimates for the compatible system holds 
as in \cite[Section~4]{Fujita-Furuta-Yoshida2} 
by considering the pull-back of $D_{\alpha}$\rq s 
by $q_L$ in next Lemma~\ref{covering}. 

\begin{lemma}\label{covering}
Suppose that $\SF$ is a foliation on a manifold $X$ and $L$ is a leaf with finite holonomy group. Take a small open tubular neighborhood $V_L$ of $L$ which is a union of leaves. Then, there is a finite covering $q_L\colon \tilde V_L\to V_L$ whose covering transformation is given by the holonomy representation. Moreover, such covering is unique up to isomorphism. 
\end{lemma}

\begin{remark}\label{characterization of q}
By taking and fixing a point $x\in L$ we can obtain a holonomy representation of the fundamental group $\pi_1(L,x)$ of $L$ with the base point $x$. The covering $q_L\colon \tilde V_L\to V_L$ is constructed by using this holonomy representation. By construction the induced foliation on $\tilde V_L$ is a bundle foliation. The covering $q_L\colon \tilde V_L\to V_L$ in Lemma~\ref{covering} is characterized by the following condition: The foliation on $\tilde V_L$ is a bundle foliation and a generic leaf in $V_L$ is diffeomorphic to some leaf in $\tilde V_L$ by $q_L$.  
%
\end{remark}
\begin{remark}
Suppose that in Lemma~\ref{covering} all holonomy groups of $\SF$ are finite. By Lemma~\ref{covering}, on a neighborhood of each leaf $L$ the covering $q_L$ is determined up to isomorphism. But, these coverings depend on the choices of the base points of the holonomy representations and are not determined canonically. In particular these coverings are not necessarily patched together globally. 
%
\end{remark}

We will use the notations $V_L$, $\tilde V_L$, and $q_L$ in the following definition. We also denote the projection map of the fiber bundle structure by $\pi_L\colon \tilde V_L\to\tilde U_L$. 

\begin{definition}\label{strongly acyclic}
A compatible system $\{ D_\alpha\}_{\alpha \in A}$ on $(\{V_{\alpha}, \SF_{\alpha}\}, W)$ is said to be {\it acyclic} if 
it satisfies the following conditions.  
\begin{enumerate}
\item The Dirac type operator 
$q^*_{L}D_{\alpha}|_{\pi_L^{-1}(\tilde b)}$ has zero kernel 
for each $\alpha\in A$, leaf $L\in \SF_{\alpha}$ and 
$\tilde b\in \tilde U_{L}$. 
\item If $V_\alpha\cap V_\beta\neq \emptyset$, then the anti-commutator 
$\{D_{\alpha},D_{\beta}\}$ is a non-negative operator on $V_\alpha\cap V_\beta$. 
\end{enumerate}
\end{definition}


\subsection{$G$-acyclic compatible system}

Suppose that there is an action of a compact Lie group $G$ on $V$ which preserves all the data. 

\begin{definition}\label{equivariant compatible system}
A compatible system $\{ D_\alpha\}_{\alpha \in A}$ on $(\{ V_\alpha ,\SF_\alpha\}_{\alpha \in A},W)$ is said to be {\it $G$-equivariant} if it satisfies the following conditions. 
\begin{enumerate}
\item $\{ V_\alpha ,\SF_\alpha \}_{\alpha \in A}$ is $G$-equivariant. 
\item For each $\alpha\in A$ $D_\alpha$ commutes with the $G$-action on $\Gamma (W|_{V_\alpha})$ given by pull-back. 
\end{enumerate}
\end{definition}

We introduce the notion of a $G$-acyclic compatible system. 
Suppose that $\{ V_\alpha ,\SF_\alpha\}_{\alpha \in A}$ is $G$-equivariant. 
For each $\alpha\in A$ let $G_\alpha$ be the subgroup of $G$ consisting of the elements preserving each leaf of $\SF_\alpha$.


\begin{lemma}
Let $L_\alpha$ be a leaf of $\SF_\alpha$. Let $V_{L_\alpha}$, $\tilde V_{L_\alpha}$, and $q_{L_\alpha}$ be the data as in Lemma~\ref{covering}. Then, $q_{L_\alpha}\colon \tilde V_{L_\alpha}\to V_{L_\alpha}$ has the unique structure of $G_\alpha$-equivariant covering such that for any generic leaf $L_\alpha'\subset V_{L_\alpha}$, the $G_\alpha$-action preserves $q_\alpha^{-1}(L_\alpha')$, and the diffeomorphism $q_{L_\alpha}|_{q_{L_\alpha}^{-1}(L_\alpha')}\colon q_{L_\alpha}^{-1}(L_\alpha')\to L_\alpha'$ is $G_\alpha$-equivariant. 
\end{lemma}
\begin{proof}
From Remark~\ref{characterization of q}, $\tilde V_{L_\alpha}$ is identified with the fiber product of the covering $L_{\alpha}'\subset V_{L_\alpha}\to L_\alpha$ and the projection $V_{L_\alpha}\to L_\alpha$. The $G_\alpha$-action is constructed using the $G_\alpha$-action on $L_\alpha'$ and $V_{L_\alpha}$. 

\end{proof}

\begin{definition}\label{G-acyclic compatible system}
Let $\{ D_\alpha\}_{\alpha \in A}$ a $G$-equivariant compatible system on $(\{ V_\alpha ,\SF_\alpha\}, W)$. $\{ D_\alpha\}_{\alpha \in A}$ is said to be {\it  $G$-acyclic} if it satisfies the following conditions. 
\begin{enumerate}
\item The $G_\alpha$-invariant part $\ker (q_L^*D_\alpha|_{\pi_L^{-1}(\tilde b)})^{G_\alpha}$ is trivial for each $\alpha\in A$, leaf $L\in \SF_{\alpha}$ and $\tilde b\in \tilde U_{L}$.
\item If $V_\alpha\cap V_\beta\neq \emptyset$, 
then the anti-commutator $\{D_{\alpha},D_{\beta}\}$ 
restricted on $\Gamma (W|_{V_{\alpha}\cap V_{\beta}})^G$ is a non-negative operator over $V_{\alpha}\cap V_{\beta}$.
\end{enumerate}
\end{definition}

\begin{remark}\label{strongly -> G-acyclic}
A $G$-equivariant acyclic compatible system is $G$-acyclic. 
\end{remark}

\section{Equivariant local index}

In this section we recall the definition of the {\it local index} 
in \cite{Fujita-Furuta-Yoshida2} and 
its equivariant version. We introduce two equivariant versions. 
We give the definition under the assumption of local torus action. 
Though the torus action itself is not essential to define the local index, 
it is enough to have an application to the quantization conjecture. 
In fact it is possible to define the local index and its equivariant version 
without using torus action under some technical assumptions. 
See \cite[Subsection~7.3]{Fujita-Furuta-Yoshida2} for details. 

Let $M$ be a Riemannian manifold which is not necessarily compact 
and $W$ a $\Z/2$-graded $Cl(TM)$-module bundle with the Clifford 
multiplication $c$. 
Suppose that there exists an open subset $V$ of $M$ 
whose complement $M\setminus V$ is compact. 
Let $G$ be a compact torus  which acts on $V$ in an isometric way, 
and the action lifts to $W|_V$ so that it commutes with $c$. 
We assume the following conditions. 
\begin{assumption}
\begin{enumerate}
\item Each $G$-orbit has positive dimension. 
\item $A$ is a finite set. 
\end{enumerate}
\end{assumption}
As we proved in Lemma~\ref{good open covering} and Proposition~\ref{rem good open covering} there exists a structure of 
a compatible fibration $\{V_{\alpha}, \SF_{\alpha}\}_{\alpha\in A}$ on $V=\cup_{\alpha\in A} V_{\alpha}$. Suppose that there exists a compatible system $\{D_{\alpha}\}_{\alpha\in A}$ on $(\{V_{\alpha}, \SF_{\alpha}\}_{\alpha\in A},W)$. 

\begin{remark}
In the next section we construct such a compatible system 
using torus action on an almost Hermitian manifold equipped with a 
Hermitian line bundle with connection. 
Under a suitable assumption of vanishing of parallel sections the compatible system is ($G$-)acyclic.  
\end{remark}

\subsection{Definition of $\ind(M,V,W)$}\label{ind(M,V,W)}
We recall the definition of the local index briefly. 
We extend the argument in \cite[Section~4]{Fujita-Furuta-Yoshida2} to our generalized setting. 
Since the argument here is parallel we omit proofs. See \cite[Section~4]{Fujita-Furuta-Yoshida2} for more details. 
We take a smooth function $f:M\to \R$ which is $G$-invariant on the end of $V$ 
and a regular value $c\gg 1 $ of $f$ such that $f^{-1}((-\infty, c])$ is compact 
and contains $M\setminus V$. 
Then we consider the complete manifold 
$\hat M=f^{-1}((-\infty, c])\cup \left( f^{-1}(c)\times [0,\infty)\right)$ 
with the cylindrical end $\hat V=(f^{-1}((-\infty, c])\cap V)\cup \left( f^{-1}(c)\times [0,\infty)\right)$ equipped with natural $G$-action. 
We can also construct a $\Z/2$-graded Clifford module bundle $\hat W\to \hat M$. See \cite{Fujita-Furuta-Yoshida1} for example for the 
construction of the Clifford action on $\hat W$. 
Since $f^{-1}(c)$ is an admissible subset in the sense of Definition~\ref{admissible subset},  the compatible fibration $\{V_{\alpha}, \SF_{\alpha}\}_{\alpha\in A}$ 
can be extended to a compatible fibration $\{\hat V_{\alpha}, \hat\SF_{\alpha}\}_{\alpha\in A}$ 
on $\hat V=\cup_{\alpha\in A}\hat V_{\alpha}$. 
Since $D_{\alpha}$ and the anti-commutator $D_{\alpha}D_{\beta}+D_{\beta}D_{\alpha}$ are operators along $G$-orbit for all $\alpha,\beta\in A $, 
we have the following. 

\begin{proposition}
The compatible system $\{D_{\alpha}\}_{\alpha\in A}$ 
can be extended to a compatible system $\{\hat D_{\alpha}\}_{\alpha\in A}$ 
on $(\{\hat V_{\alpha}, \hat\SF_{\alpha}\}_{\alpha\in A},W)$, which has translationally invariance on the 
end. 
Moreover if $\{D_{\alpha}\}_{\alpha\in A}$ is acyclic, then 
$\{\hat D_{\alpha}\}_{\alpha\in A}$ is also acyclic. 
\end{proposition}

We can take a $G$-invariant partition of unity $\{\rho_{\alpha}^2\}_{\alpha\in A}$ 
subordinate to $\{V_{\alpha}\}$ with some specific estimate as in \cite[Section~4]{Fujita-Furuta-Yoshida2}. 
Let $D$ be a Dirac-type operator on $\Gamma(\hat W)$ which is translationally invariant on the end. 
For $t\geq 0$ we consider the perturbation 
$$
D_t:=D+t\sum_{\alpha\in A}\rho\hat D_{\alpha}\rho_{\alpha}. 
$$
Note that $D_t$ has a decomposition into even-part and odd-part, 
$D_t=D_t^0+D_t^1$. 
The following is a main result in \cite{Fujita-Furuta-Yoshida2}. 
(See \cite[Theorem~7.2]{Fujita-Furuta-Yoshida2}.) 

\begin{theorem}
If $\{D_{\alpha}\}_{\alpha\in A}$ is acyclic, then 
for $t\gg 0$ the space of $L^2$-solutions of $D_ts=0$ is a $\Z/2$-graded finite dimensional vector space. 
Its super-dimension is independent of a sufficiently large $t\gg 0$ and any other continuous deformations of data. 
\end{theorem}

\begin{definition}\label{def of ind(M,V,W)}
We define the {\it local index} $\ind(M,V,W)=\ind (M,V,W,\{ V_\alpha ,\SF_\alpha\} , \{ D_\alpha\})$ to be the super-dimension,  
$$
\ind(M,V,W):=\dim \ker D_t^0\cap L^2(\hat W) -\dim \ker D_t^1\cap L^2(\hat W). 
$$
\end{definition}

\begin{proposition}[Proposition~7.3 and Proposition~7.4 in \cite{Fujita-Furuta-Yoshida2}]
The local index $\ind(M,V,W)$ does not depend on the various choices of data. 
In particular it does not depend on the completion $\hat M$. 
\end{proposition}

\begin{remark}
The local index $\ind(M,V,W)$ depends on the local torus action on $V$ 
and the choice of open covering $\{V_{\alpha}\}_{\alpha\in A}$. 
\end{remark}

\begin{theorem}[\cite{Fujita-Furuta-Yoshida2}]\label{local index}
The local index $\ind (M,V,W)$ has the following properties. 
\begin{enumerate}
\item $\ind (M,V,W)$ is invariant under continuous deformations of the data. 
\item \label{ind(M,V)=ind D}If $M$ is closed, then $\ind (M,V,W)$ is equal to the index $\ind D$ of a Dirac-type operator $D$ on $W$. 
\item If $V'$ is an admissible open subset of $V$ with complement $M\setminus V'$ compact, then we have 
\[
\ind (M,V,W)=\ind (M,V',W). 
\]
\item If $M'$ is an open neighborhood of $M\setminus V$ with $V\cap M'$ admissible, then $\ind(M,V,W)$ has the following excision property
\[
\ind(M,V,W)=\ind (M',V\cap M',W|_{M'}). 
\]
\item \label{disjoint union}Suppose $M$ is a disjoint union $M=M_1\coprod M_2$. 
Then we have the following sum formula
\[
\ind (M,V,W)=\ind(M_1,V\cap M_1,,W|_{M_1})+\ind(M_2,V\cap M_2,W|_{M_2}). 
\]
\item We have a product formula for $\ind (M,V,W)$. For the precise statement see \textup{\cite[Theorem~8.8]{Fujita-Furuta-Yoshida2}}. See also Section~\ref{product formula}. 
\end{enumerate}
\end{theorem}

In the rest of this paper, for simplicity, we sometimes use the notation $\ind (M,W)$ for $\ind (M,V,W)$. It would be no confusion since $\ind (M,V,W)$ have the excision property. 

\begin{remark}
From the properties~\ref{ind(M,V)=ind D} and \ref{disjoint union} in Theorem~\ref{local index} we can show that $\ind (M,V,W)$ also has the following vanishing property. 
\begin{enumerate}
\item[{\it 7}.] {\it If $M=V$, then we have}
\[
\ind(M,V,W)=0. 
\]
\end{enumerate}
\end{remark}

To obtain the product formula we need to formulate and 
define $\ind(M,V,W)$ for a manifold whose end is the total space of 
a fiber bundle such that both of its base space and its fiber are 
manifolds with cylindrical end. 
A similar generalization is necessary for $\ind^G(M,V,W)$ 
which will be defined in 
Subsection~\ref{G-acyclic system and invariant part}. 
For more details, see \cite{Fujita-Furuta-Yoshida2}.  

As a corollary of the excision formula in Theorem~\ref{local index} we have a localization formula of index. 
\begin{corollary}\label{localization1}
Under the assumption of Theorem~\ref{local index}, suppose that there exists an open covering $\{ O_i\}_{i=1}^m$ of $M\setminus V$ which satisfies the following properties. 
\begin{enumerate}
\item $\{O_{i}\}_{i=1}^m$ are mutually disjoint. 
\item Each $O_i$ is admissible. 
\end{enumerate}
Then we have 
\[
\ind (M,V,W)=\sum_{i}\ind(O_i,O_i\cap V,W_{O_i}). 
\]
\end{corollary}

\subsection{$K$-equivariant local index $\ind_K(M,V,W)$}\label{ind_K(M,V)}
Let $K$ be a compact Lie group which acts on $M$ in an isometric way and 
the action lifts to $W$ so that it commutes with the Clifford multiplication. 
Suppose that $V$ is $K$-invariant and $K$-action commutes with $G$-action on $V$. 
Then by Proposition~\ref{rem good open covering}, 
$\{V_{\alpha}, \SF_{\alpha}\}_{\alpha\in A}$ is a $K$-equivariant compatible fibration. 
We also assume that the compatible system $\{D_{\alpha}\}_{\alpha\in A}$ is $K$-equivariant. 
We can take a completion $\hat M$, $\{\hat V_{\alpha}\}_{\alpha\in A}$ etc. in Subsection~\ref{ind(M,V,W)} in a $K$-equivariant way, and we can consider the perturbation 
$$
D_{t}=D+t\sum_{\alpha\in A}\rho_{\alpha}\hat D_{\alpha}\rho_{\alpha}, 
$$which is also $K$-equivariant.  

\begin{theorem}
If $\{D_{\alpha}\}_{\alpha\in A}$ is acyclic, then 
for $t\gg 0$ the space of $L^2$-solutions of $D_ts=0$ becomes a $\Z/2$-graded finite dimensional $K$-representation. 
It defines a virtual $K$-representation, which is independent of a sufficiently large $t\gg 0$ and any other continuous deformations of data. 
\end{theorem}

\begin{definition}
We define the {\it $K$-equivariant local index} $\ind_K(M,V,W)=\ind_K(M,V,W,\{ V_\alpha ,\SF_\alpha\} , \{ D_\alpha\})$ of the $K$-equivariant acyclic system to be the virtual $K$-representation 
$$
\ind_K(M,V,W):=\ker D_t^0\cap L^2(\hat W) -\ker D_t^1\cap L^2(\hat W)\in R(K),  
$$where $R(K)$ is the representation ring of $K$. 
\end{definition}
As for the local index, the $\ind_K(M,V,W)$ does not depend on the various choices. 
For these data, we obtain an equivariant version of \cite[Theorem~7.9]{Fujita-Furuta-Yoshida2} for $\ind_K (M,V,W)$. 
\begin{theorem}\label{equivariant localization}
Suppose that there exists an open covering $\{ O_i\}_{i=1}^m$ of $M\setminus V$ which satisfies the following properties. 
\begin{enumerate}
\item $\{O_{i}\}_{i=1}^m$ are mutually disjoint. 
\item Each $O_i$ is $K$-invariant. 
\item Each $O_i$ is admissible.
\end{enumerate}
Then we have the following localization formula for $\ind_K(M,V,W)$
\begin{equation}\label{equivariant localization formula}
\ind_K(M,V,W)=\sum_{i}\ind_K(O_i,O_i\cap V,W|_{O_i})\in R(K).  
\end{equation}
\end{theorem}


\subsection{$G$-invariant local index $\ind^G(M,V,W)$}\label{G-acyclic system and invariant part}
In this section we assume that $G$ acts on whole $M$ preserving all the data.  
We can take a completion $\hat M$, $\{\hat V_{\alpha}\}_{\alpha\in A}$ etc. in Subsection~\ref{ind(M,V,W)} in a $G$-equivariant way, and we can consider the perturbation 
$$
D_{t}=D+t\sum_{\alpha\in A}\rho_{\alpha}\hat D_{\alpha}\rho_{\alpha}, 
$$which is also $G$-equivariant.  
Then, the same argument as that used to define $\ind(M,V,W)$ in Subsection~\ref{ind(M,V,W)} holds for the $G$-invariant part of $\ker D_t$ under the assumption of $G$-acyclicity. 

\begin{theorem}
If $\{D_{\alpha}\}_{\alpha\in A}$ is $G$-acyclic, then 
for $t\gg 0$ the space of $G$-invariant $L^2$-solutions of $D_ts=0$ is a $\Z/2$-graded finite dimensional vector space. 
Its super-dimension is independent of a sufficiently large $t\gg 0$ and any other continuous deformations of data. 
\end{theorem}

\begin{definition}
We define the {\it $G$-invariant local index} $\ind^G(M,V,W)=\ind^G(M,V,W,\{ V_\alpha ,\SF_\alpha\} , \{ D_\alpha\})$ of the $G$-acyclic compatible system to be the super-dimension  
$$
\ind^G(M,V,W):=\dim(\ker D_t^0)^G\cap L^2(\hat W) -\dim(\ker D_t^1)^G\cap L^2(\hat W)\in\Z.   
$$ 
\end{definition}
As for the local index, the $\ind^G(M,V,W)$ does not depend on the various choices. 
For $\ind^G(M,V,W)$ we have the following localization formula.  
\begin{theorem}\label{localization for invariant part}
Under the above assumption 
suppose that there exists an open covering $\{ O_i\}_{i=1}^m$ of $M\setminus V$ which satisfies the following properties. 
\begin{enumerate}
\item $\{O_{i}\}_{i=1}^m$ are mutually disjoint. 
\item Each $O_i$ is $G$-invariant. 
\item Each $O_i$ is admissible. 
\end{enumerate}
Then we have 
\[
\ind^G(M,V,W)=\sum_{i}\ind^G(O_i,O_i\cap V,W|_{O_i}). 
\]
\end{theorem}
\begin{remark}
Since a $G$-equivariant acyclic compatible system is $G$-acyclic (see Remark~\ref{strongly -> G-acyclic}), the $G$-equivariant acyclic compatible system $\{D_\alpha\}_{\alpha\in A}$ in Theorem~\ref{equivariant localization} can be taken as the $G$-acyclic compatible system $\{D_\alpha\}_{\alpha\in A}$ in Theorem~\ref{localization for invariant part}. In this case, Theorem~\ref{localization for invariant part} can be obtained from Theorem~\ref{equivariant localization} for $K=G$ by taking the $G$-invariant part of $\ind_G(M,V,W) $. 
However, for a general $G$-acyclic compatible system Theorem~\ref{localization for invariant part} is not obtained from Theorem~\ref{equivariant localization}.  
\end{remark}

\begin{remark}
Though the $G$-invariant local index can be define for general $G$-acyclic compatible system which does not come from torus action, 
we only use the case of torus action in this paper.  
\end{remark}

\subsection{$K$-equivariant $G$-invariant local index $\ind_K^G(M,V,W)$}
We combine the assumptions in Subsection~\ref{ind_K(M,V)} and Subsection~\ref{G-acyclic system and invariant part}. 
Namely the product $G\times K$ acts on $M$ in an isometric way and the action lifts to $W$ so that it commutes with the Clifford multiplication, 
and the compatible system $\{D_{\alpha}\}_{\alpha\in A}$ is $G\times K$-equivariant. 
Suppose that $\{ D_\alpha\}_{\alpha \in A}$ is $G$-acyclic. Then, we have $\ind^G(M,V,W)$ and it becomes a virtual representation of $K$. We denote it by $\ind_K^G(M,V,W)$. In this case, we have a $K$-equivariant version of Theorem~\ref{localization for invariant part}. 
\begin{theorem}\label{equivariant localization for invariant part}
Under the above assumption 
suppose that there exists an open covering $\{ O_i\}_{i=1}^m$ of $M\setminus V$ which satisfies the following properties. 
\begin{enumerate}
\item $\{O_{i}\}_{i=1}^m$ are mutually disjoint. 
\item Each $O_i$ is $G\times K$-invariant. 
\item Each $O_i$ is admissible. 
\end{enumerate}
Then we have 
\[
\ind_K^G(M,V,W)=\sum_{i}\ind_K^G(O_i,O_i\cap V,W|_{O_i}). 
\]
\end{theorem}

\section{The case of torus action}
\label{The case of torus action}
Let $G$ be a compact torus and $K$ a compact Lie group. 
Let $V$ be a smooth manifold equipped with an 
action of $G\times K$ and a $G\times K$-invariant almost complex structure $J$. We take and fix a Riemannian metric $g$ on $V$ 
which is invariant by the almost complex structure $J$ 
and the $G\times K$-action. 
Let $L\to V$ be a $G\times K$-equivariant Hermitian line bundle over $V$ 
equipped with a $G\times K$-invariant Hermitian connection $\nabla$.  
Note that since $\nabla$ is invariant under the torus action 
the restriction of $(L,\nabla)$ to each $G$-orbit 
is a flat line bundle. 
Let $A$ be the set of subgroups of $G$ which appears as the identity component 
of the stabilizer group at some point $x\in V$. 
We assume the following four conditions. 

\begin{enumerate}
\item Each $G$-orbit is totally real, i.e., 
any subspace $\xi$ of the tangent space of the orbit 
satisfies $\xi\cap J\xi=0$. 
\item Each $G$-orbit has positive dimension. 
\item $A$ is a finite set. 
\item For each $x\in V$ there exists an open neighborhood 
$V_x$ of $x$ such that $gV_x=V_{gx}$ for all $g\in G$ and 
 the restriction of $L$ to $G_x^{\perp}$-orbit 
$G_x^{\perp}y$ has no $G_x^{\perp}$-invariant 
nontrivial parallel sections for all $y\in V_x$. 

\end{enumerate}
Note that since each orbit is a torus 
the last condition is equivalent to the vanishing of the 
$G_x^{\perp}$-invariant part of 
cohomologies with local coefficient, 
$H^*(G_x^{\perp}y, L|_{G_x^{\perp}y})^{G_x^{\perp}} =0$. 
See \cite[Lemma~5.12]{Fujita-Furuta-Yoshida2} for example. 
By Proposition~\ref{rem good open covering} there exists an open covering 
$\{V_{H}\}_{H\in A}$ of $V$ which has a structure of $G\times K$-equivariant compatible fibration $\{V_H, \SF_{H}\}_{H\in A}$.  
In this section we show that $V$ is equipped with a structure of 
$K$-equivariant $G$-acyclic compatible system $\{D_{H}\}_{H\in A}$. 

\subsection{$K$-equivariant $G$-acyclic compatible system on $(\{V_H, \SF_H\}_{H\in A},W)$}
\label{torus : compatible system}
Let $T\SF_{H}\to V_{H}$ be the tangent bundle along leaves of $\SF_{H}$. 
Since each $G$-orbit is totally real 
we have a canonical injection 
$T\SF_{H}\otimes_{\R} \C\to TV_{H}$ via 
the almost complex structure on $V$. 
Let $T'V_{H}$ be the orthogonal complement of 
$T\SF_{H}\otimes_{\R} \C$ in the Hermitian vector bundle $TV_{H}$, 
which is canonically isomorphic to $TV_{H}/(T\SF_{H}\otimes \C)$. 
There is a canonical isomorphism 
$TV_{H}\cong (T\SF_{H}\otimes \C)\oplus T'V_{H}$ 
as Hermitian vector bundles. Using the isomorphism we have an isomorphism 
$$
\wedge^{0,\bullet} T^*V_{H}\cong 
(\wedge^{\bullet} T^*\SF_{H})\otimes 
 \wedge^{0,\bullet}(T'V_{H})^*. 
$$
We define a $\Z/2$-graded Clifford module bundle $W$ by 
$$
W:=\wedge^{0,\bullet}T^*V\otimes L.  
$$
Note that there is the canonical isomorphism 
$$
W|_{V_{H}}\cong (\wedge^{\bullet} T^*\SF_{H})\otimes 
\wedge^{0,\bullet} (T'V_{H})^*\otimes L|_{V_{H}}. 
$$
Let $D_{H}:\Gamma(W|_{V_{H}})\to \Gamma(W|_{V_{H}})$ be the 
Dirac operator along leaves of $\SF_{H}$ 
which is defined by the de Rham operator with coefficient in $L|_{V_{H}}$. Strictly speaking $D_{H}$ is a differential operator acting on 
$\Gamma((\wedge^{\bullet} T^*\SF_{H})\otimes L|_{V_{H}})$ 
which is defined by the de Rham operator along leaves of $\SF_{H}$ and 
the Hermitian connection $\nabla$ of $L$. 
Since the restriction of $T'V_{H}$ to each leaf of $\SF_{H}$ 
has a canonical flat structure induced by the $H^{\bot}$-action, 
we can regard $D_{H}$ as a differential operator 
acting on $\Gamma((\wedge^{\bullet} T^*\SF_{H})\otimes 
\wedge^{0,\bullet} (T'V_{H})^*\otimes L|_{V_{H}})$.  
Then $\{D_H\}_{H\in A}$ is a $G\times K$-equivariant compatible system 
on $\{V_{H}, \SF_H\}_{H\in A}$.

\begin{proposition}\label{G-acyclicity}
The data $\{D_{H}\}$ is a 
$K$-equivariant $G$-acyclic compatible system on $V$. 
\end{proposition}
\begin{proof}

Suppose that $x\in V$ is contained in $V_H$ for some $H\in A$. 
Let $H^{\perp}x$ be the $H^{\perp}$-orbit through the point $x$ 
and $\tilde U_x$ the slice at $x$ with respect to the $H^{\perp}$-action. 
We put $\tilde V_x:=H^{\perp}\times \tilde U_x$. 
Note that the natural map  
$q_x:\tilde V_x\to V_{x}:=H^{\perp}\tilde U_x$ 
is a finite covering whose covering transformation group is 
equal to $G_x\cap H^{\perp}$. 
Let $\pi_x:\tilde V_x\to \tilde U_{x}$ be the projection 
to the second factor. 
The finite covering $q_x:\tilde V_x\to V_x$ and 
the $H^{\perp}$-bundle $\pi_{x}:\tilde V_x\to \tilde U_x$ are
the data in Lemma~\ref{covering}.  
In the above setting $H^{\perp}$ is the subgroup 
which preserves each leaf of $\SF_H$. 
By the 4th condition in our assumption and property~(4) in Lemma~\ref{good open covering} 
there exists $x'\in V$ such that $x\in V_{x'}$, $G_{x'}=H$ and 
$(\ker D_H)^{H^{\perp}}=H^*(G_{x'}^{\perp}x; L|_{G_{x'}})^{G_{x'}^{\perp}}=0$. 
Moreover if $V_H\cap V_{H'}\neq\emptyset$ and $H'\subset H$, 
then the anti-commutator $\{D_H,D_{H'}\}$ is the 
Laplacian on the $H^{\perp}$-orbit (Lemma~5.9 
in \cite{Fujita-Furuta-Yoshida2}), and hence, it is non-negative.
Then we complete the proof. 
\end{proof}

\subsection{Cotangent bundle of the torus}
\label{cotangent bundle case}
Let $T^*G$ be the cotangent bundle of the torus $G$. 
Consider the canonical symplectic form $\omega$ on $T^*G$ 
defined by the canonical 1-form $\alpha$ on $T^*G$. 
Note that $T^*G$ has a natural $G$-action induced by 
the multiplication of $G$. 
We fix a $G$-equivariant trivialization $T^*G\cong G\times \g^*$, 
where $\g^*$ is the dual of the Lie algebra of $G$. 
Let $L$ be the $G$-equivariant Hermitian line bundle 
$T^*G\times \C$ over $T^*G$ 
with Hermitian connection $\nabla$ defined by $\nabla=d-\sqrt{-1}\alpha$, 
where the $G$-action on the $\C$-factor is trivial.  
Let $\g_{\Z}$ be the integral lattice and $\g^*_\Z$ the weight lattice, i.e., 
\[
\begin{split}
\g_{\Z}&=\{ \xi \in \g \mid \exp \xi =e\in G\} ,\\ 
\g^*_\Z&=\{ \eta^*\in \g^* \mid \< \xi,\eta^*\>\in \Z\ \forall \xi \in \g_{\Z}\} . 
\end{split}
\]
Take an integral weight $\xi \in \g^*_{\Z}$ and 
define $L(\xi)$ by $L(\xi)=L\otimes \underline{\xi}$, 
where $\underline{\xi}$ is the trivial line bundle equipped with 
the $G$-action whose weight is given by $\xi$. 
The definition of the canonical $1$-form implies that 
$\eta\in\g^*$ is in weight lattice $\g_{\Z}^*$ if and only if 
the restriction of $(L,\nabla)$ to $G\times \{\eta\}$ is trivially flat. 
Moreover the proof of \cite[Lemma~5.12]{Fujita-Furuta-Yoshida2} shows 
the following lemma. 
\begin{lemma}\label{G-cohomology}
We have the following isomorphism between representations of $G$. 
$$
H^*(G\times\{\eta\}, (L(\xi),\nabla)|_{G\times\{\eta\}})\cong
\left\{\begin{array}{ll}
0  \quad (\eta\notin \g_{\Z}^*) \\ 
H^*(G\times\{\eta\}, \C)\otimes \xi \quad (\eta\in \g_{\Z}^*). 
\end{array}\right.
$$
\end{lemma} 
Let $M$ be a small $G$-invariant open neighborhood 
of the zero section in $T^*G$ so that 
the image of $M$ under the projection $T^*G\to \g^*$ 
does not contain any non-zero integral points. 
Let $V$ be the complement of the zero section in $M$. 
Consider the natural complex structure on $T^*G$ 
induced from the trivialization of $T^*G$ and the metric of $\g$, 
which is compatible with the symplectic structure. 
By Lemma~\ref{G-cohomology} and the construction in the previous subsections 
we have a $G$-equivariant acyclic compatible system on $V$ for the Clifford module bundle $W(\xi):=\wedge^{0,\bullet}T^*M\otimes L(\xi)$, 
and we can define the $G$-equivariant index ${\rm ind}_G(M,V,W(\xi))\in R(G)$. 
Note that since the $G$-action on $T^*G$ is free, 
we use the open covering consisting of the single open set $V$. 
In \cite[Remark~6.10]{Fujita-Furuta-Yoshida1} we give explicit solutions 
of the equation $D_ts=0$ in the case of $\dim G=1$.  
Using the explicit description we have the following, 
which will be used in the proof of Theorem~\ref{equiv[Q,R]=0} in Subsection~\ref{Proof of S^1-case}. 
\begin{proposition}\label{index of T^*S^1}
If $G$ is the circle group $S^1$, then we have 
${\rm ind}_G(M,V,W(\xi))=\xi$. 
\end{proposition}

\begin{remark}

It is expected that a similar argument is possible to
calculate  the equivariant index for higher dimensional cases.
The numerical index is already calculated in our previous paper
\cite[Theorem~6.11]{Fujita-Furuta-Yoshida1} and is equal to $1$.
In the calculation there we used the embedding
$T^*G\subset G\times G$ derived from the one-point compactification
$\R \to S^1$.  This compactification, however, is not $G$-equivariant
and hence is not available to the calculation of the equivariant index.
Another possible approach to the higher dimensional case
would be to use the product structure
$G=(S^1)^n$ and apply the product formula of equivariant index. 
However,
it would be necessary to compare
the compatible system given by the product structure
with the one given by the $G$-action on $T^*G$.
Since we have not shown such comparizon,  
this approach is not completed yet.
Because the $G$-action on $H^*(G,\C)$ is trivial,
we have at least the next vanishing property of the $G$-invariant part
from Lemma~\ref{G-cohomology} and 
the vanishing of $G$-invariant index.
\end{remark}

\begin{proposition}
For $\xi\neq 0$ we have ${\rm ind}^G(M,V,W(\xi))=0$. 
\end{proposition}

\section{Vanishing theorem for $S^1$-acyclic compatible systems}

In this section we show the vanishing theorem 
of $\ind_K^{S^1}(M,V,W)$ for $K$-equivariant $S^1$-acyclic compatible system 
under the setting in Section~\ref{The case of torus action}. 

Let $K$ be a compact Lie group. 
Let $M$ be a smooth manifold equipped with an 
action of $S^1\times K$ and $S^1\times K$-invariant almost complex structure. 
We take and fix an $S^1\times K$-invariant Hermitian metric on $M$. 
Suppose that the fixed point set $M^{S^1}$ of the $S^1$-action is 
a closed connected submanifold of $M$. 
Let $L\to M$ be an $S^1\times K$-equivariant Hermitian line bundle over $M$ 
equipped with an $S^1\times K$-invariant Hermitian connection $\nabla$ 
such that the fixed point set $L^{S^1}$ of the $S^1$-action 
is equal to the image of the zero section of $M^{S^1}$ to $L|_{M^{S^1}}$. 
Note that the restriction of $(L,\nabla)$ to each $S^1$-orbit 
is a flat line bundle. 

In the next subsection we show that 
there is an $S^1\times K$-invariant open neighborhood $M'$ of $M^{S^1}$
and a $K$-equivariant $S^1$-acyclic compatible system on $M'\setminus M^{S^1}$.  Here we use the Clifford module bundle $W:=\wedge^{0,\bullet}T^*M'\otimes L$. 
The main theorem in this section is the following vanishing theorem. 
\begin{theorem}\label{main S^1-vanishing}
$$
{\rm ind}_K^{S^1}(M',M'\setminus M^{S^1}, W)=0\in R(K). 
$$
\end{theorem}


\subsection{$S^1$-acyclic compatible system}
\label{$S^1$-acyclic compatible system} 

In this subsection we show that 
there is an open subset of $M$ on which we have 
a $K$-equivariant $S^1$-acyclic compatible system. 
To show it we first show the following.

\begin{lemma}\label{vanishing of cohomology}
For each $x\in M$ let $(L_{(x)}, \nabla_{(x)})$ be the restriction 
of the pull-back of $(L,\nabla)$ to $S^1\times\{x\}$ 
by the multiplication map $S^1\times M\to M$.  
There exists an $S^1\times K$-invariant open neighborhood $M'$ of $M^{S^1}$ such that 
for a each $x\in M'$ the $S^1$-invariant part of the de Rham cohomology 
with local coefficient $H^*(S^1\times \{x\};(L_{(x)}, \nabla_{(x)}))^{S^1}$ is zero. 
\end{lemma}

\begin{proof}
For each $x\in M^{S^1}$ we have the canonical isomorphism 
$$
H^*(S^1\times \{x\};(L_{(x)}, \nabla_{(x)}))\cong H^*(S^1; \C)\otimes L_{(x)}, 
$$and the $S^1$-invariant part of the right hand side is zero 
because $L_{(x)}$ is a non-trivial representation of $S^1$ by our assumption. 
By the semi-continuity of the cohomology 
there exists an $S^1\times K$-invariant open neighborhood $M'$ of $M^{S^1}$ such that 
$H^*(S^1\times \{x\};(L_{(x)}, \nabla_{(x)}))^{S^1}$ is zero for all $x\in M'$. 
\end{proof}

Using the open subset $M'$ obtained in 
Lemma~\ref{vanishing of cohomology} we put $V:=M'\setminus M^{S^1}$. 
There exist the structure of a compatible fibration on $V$ and 
a $K$-equivariant $S^1$-acyclic compatible sytem on it
as in Section~\ref{The case of torus action}. 
\subsection{Product formula}\label{product formula}
In this subsection we recall a product formula of local indices, 
which we will use in the proof of Theorem~\ref{main S^1-vanishing}. 

Let $Y_0$ be a closed manifold equipped with a $K$-action and 
a $K$-invariant Hermitian structure. 
Let $L_0\to Y_0$ be a $K$-equivariant Hermitian line bundle with 
a $K$-invariant Hermitian connection. 
Define $W_0$ to be $W_0:=\wedge^{0,\bullet}T^*Y_0\otimes L_0$, 
which has a natural structure of a $K$-equivariant $\Z/2$-graded 
Clifford module bundle over $Y_0$. 
Let $K_1$ be a compact Lie group and $Q\to Y_0$ a 
$K$-equivariant principal $K_1$-bundle over $Y_0$. 
Let $Y_1$ be a unitaty representation of $K_1\times S^1$ 
such that $Y_1^{S^1}=\{0\}$. 
Let $R$ be a $1$-dimansional unitary representation of $S^1$ and 
$L_1$ the $K_1\times S^1$-equivariant line bundle $Y_1\times R\to Y_1$. 
Here $K_1$ acts on $R$ trivially 
and we consider the product connection on $L_1$. 
Define $W_1$ by $W_1:=\wedge^{0,\bullet}T^*Y_1\otimes L_1$, 
which has a natural structure of $K_1\times S^1$-equivariant $\Z/2$-graded 
Clifford module bundle over $Y_1$. 
We put $Y:=Q\times_{K_1}Y_1$ and $W:=W_0\otimes (Q\times_{K_1}W_1)$. 

Note that $Y_1$ has a structure of a $K$-equivariant $S^1$-acyclic compatible system by taking $M'=Y_1$ and $M^{S^1}=\{0\}$ 
in Lemma~\ref{vanishing of cohomology}. 
Using the product formula~\cite[Theorem~8.8]
{Fujita-Furuta-Yoshida2},  we have 
the following equality. 
\begin{proposition} \label{product formula for fixed point set}
$$
{\rm ind}_{K}^{S^1}(Y,W_Y)={\rm ind}_{K}(Y_0, W_0 \otimes 
(Q \times_{K_1} {\rm ind}_{K_1}^{S^1}(Y_1,W_1))). 
$$
\end{proposition}

\begin{remark}
The meaning of the right hand side in the above equality is as follows. 
As a character of $K_1$ we write ${\rm ind}_{K_1}^{S^1}(Y_1,W_1)$ as 
${\rm ind}_{K_1}^{S^1}(Y_1,W_1)=[F_0]-[F_1]$, where $F_0$ and $F_1$ are 
finite dimensionl representations of $K_1$. 
Then we put 
$$
{\rm ind}_{K}(Y_0, W_0 \otimes 
(Q \times_{K_1} {\rm ind}_{K_1}^{S^1}(Y_1,W_1))):=
{\rm ind}_{K}(Y_0, W_0 \otimes (Q \times_{K_1} F_0))-
{\rm ind}_{K}(Y_0, W_0 \otimes (Q \times_{K_1} F_1)). 
$$
\end{remark}

\subsection{Model of the neighborhood of $M^{S^1}$}
Now we come back to the setting in Subsection~\ref{$S^1$-acyclic compatible system}. 
Let $\nu \to M^{S^1}$ be the normal bundle of $M^{S^1}$ in $M'$. 
Then the fibers of $\nu$ are unitary representation of $S^1$. 
Since we assume that $M^{S^1}$ is connected 
they are mutually isomorphic unitary representations. 
We take and fix a copy $R_{\nu}$ of 
the unitary representation of $S^1$ on $\nu$. 
As in the same way we take and fix a copy $R_L$ of the 
one-dimensional unitary representation of $S^1$ on $L|_{M^{S^1}}$.  

Let $K_1$ be the group of unitary transformations of $R_{\nu}$ 
which commute with the $S^1$-action. 
Let $Q\to M^{S^1}$ be the $K$-equivariant principal $K_1$-bundle 
whose fiber $Q_x$ at $x\in M^{S^1}$ is defined by 
the set of isomorphisms between $R_{\nu}$ and $\nu_x$ as $S^1$-representations. 
Note that $\nu$ is equal to the associated vector bundle $Q\times_{K_1}R_{\nu}$.
Let $L_0$ be the $K$-equivariant Hermitian line bundle 
with connection over $M^{S^1}$ defined by 
$$
L_0:={\rm Hom}_{S^1}(R_L,L|_{M^{S^1}}). 
$$
Note that $L_0$ is abstractly isomorphic to $L|_{M^{S^1}}$ 
as Hermitian line bundles with connection  
but $L_0$ does not have $S^1$-action. 

Let $L_1$ be the $K_1\times S^1$-equivariant 
Hermitian line bundle with connection over $R_{\nu}$ defined by 
$$
L_1:=R_{\nu}\times R_{L} 
$$with the product conneciton. 
Note that $K_1$ acts trivially on the second factor. 
Let $(L_Y,\nabla_Y)$ be the Hermitian line bundle 
with connection over $\nu$ defined by 
$$
L_Y:=L_0\otimes (Q\times_{K_1}L_1)
$$
and $\nabla_Y$  the tensor product connection. 
Applying Proposition~\ref{product formula for fixed point set} 
for $Y_0=M^{S^1}$, $Y_1=R_{\nu}$, $Y=\nu$ and the associated $\Z/2$-graded 
equivariant Clifford module bundles we have 
\begin{equation}\label{model product}
{\rm ind}_{K}^{S^1}(\nu, W_{\nu})=
{\rm ind}_K(M^{S^1}, W_0\otimes (Q\times_{K_1}{\rm ind}_{K_1}^{S^1}(R_{\nu}, W_1))). 
\end{equation}

\subsection{Comparison with the model}
In this subsection we show the following 
in the setting in Subsection~\ref{$S^1$-acyclic compatible system}. 
\begin{proposition} \label{comparison}
$$
{\rm ind}_K^{S^1}(M',W|_{M'})={\rm ind}_K^{S^1}(\nu, W_{\nu}). 
$$
\end{proposition}
As as a corollary we have the following by (\ref{model product}). 
\begin{proposition}\label{product formula for fixed point set 2}
$$
{\rm ind}_K^{S^1}(M', W|_{M'})={\rm ind}_K(M^{S^1}, W_0\otimes (Q\times_{K_1}{\rm ind}_{K_1}^{S^1}(R_{\nu}, W_1))). 
$$
\end{proposition}

To show Proposition~\ref{comparison} we construct a 
one parameter family which connects two 
$K$-equivariant $S^1$-acyclic compatible systems 
on neighborhoods of $\nu^{S^1}=M^{S^1}$ in 
$(\nu, W_{\nu})$ and $(M',W|_{M'})$. 

Let $Y'$ be an $S^1\times K$-invariant tubular neighborhood of $\nu^{S^1}$ in $\nu$. 
We may assume the exponential map $\phi: Y'\to M'$ is a diffeomorphism. 
Let $g'_{M'}$ (resp. $J'_{M'}$) be the pull back of 
the Riemannanian metric $g_{M'}$ 
(resp. the almost complex structure $J_{M'}$) on $M'$ by $\phi$. 
On the other hand there is a Riemannian metric $g_{Y'}$ 
and a compatible almost complex strucutre $J_{Y'}$ 
on $\nu$ induced by the one in $TM$. 
Note that $g'_{M'}$ (resp. $J'_{M'}$) coincides with $g_{Y'}$ (resp. $J_{Y'}$) 
on the zero section $\nu^{S^1}=M^{S^1}$.  

For $t\in [0,1]$ we define a family $(g_t, J_t, L_t,\nabla_t)$ on $Y'$ 
which connects $(g'_{M'}, J'_{M'}, \phi^*L|_{M'},\phi^*\nabla|_{M'})$ and 
$(g_{Y'}, J_{Y'}, L_Y|_{Y'},\nabla_Y|_{Y'})$. 
The Riemannian metric $g_t$ is defined by $g_t:=tg'_{M'}+(1-t)g_{Y'}$. 
Note that for each $y\in Y'$ 
the set of almost complex structures of $T_yY'$ which are 
compatible with $g_t$ is a closed submanifold of ${\rm End}(TY')_y$. 
We can take $Y'$ small enough so that 
the endmorphism $J'_t:=tJ'_{M'}+(1-t)J_{Y'}$ 
is contained in a small normal disk bundle of the above 
closed submanifold with respect to the metric $g_t$, and 
there exists the unique compatible almost complex structure 
$(J_t)_y$ which minimizes the distance from $(J'_t)_y$ in ${\rm End}(TY')_y$. 
Then $J_t=\{(J_t)_y\}_{y\in Y'}$ is the reqiured 
family of almost complex structures on $Y'$ compatible with $g_t$. 
Let $\phi_t:Y'\to M'$ be the map defined by $\phi_t(v):=\phi(tv)$, 
which connects the projection $\phi_0:Y'\to M^{S^1}$ and 
the exponential map $\phi_1=\phi:Y'\to M'$. 
Using this family of maps we put 
$(L_t,\nabla_t):=\phi_t^*(L,\nabla)|_{M'}$. 
We denote the Hermitian manifold $Y'$ with the 
Hermitian structure $(g_t, J_t)$ by $Y_t$. 
Then we have a family of $S^1\times K$-equivariant $\Z/2$-graded 
Clifford module bundle $W_t$ defined by 
$W_t:=\wedge^{0,\bullet}T^*Y_t\otimes L_t$. 

\begin{lemma}\label{vanishing of cohomology for family}
For each $(y,t)\in Y\times [0,1]$ let $(L_{(y,t)},\nabla_{(y,t)})$ be 
the restriction of the pull-back of 
$(L\times [0,1],\nabla)$ to $S^1\times\{(y,t)\}$ by the multiplication map 
$S^1\times Y\times [0,1]\to Y\times [0,1]$. 
There exixts an $S^1\times K$-invariant open neighborhood 
$Y''$ of $Y^{S^1}$ in $\nu^{S^1}$ such that for each 
$(y,t)\in (Y''\setminus Y^{S^1})\times [0,1]$ 
the $S^1$-invariant part of the de Rham cohomology 
$H^*(S^1;(L_{(y,t)},\nabla_{(y,t)}))^{S^1}$ is zero. 
\end{lemma}
\begin{proof}
This lemma follows from Lemma~\ref{vanishing of cohomology} for 
$M:=Y'\times [0,1]$.  
\end{proof}

\begin{proof}[Proofof Proposition~\ref{comparison}]
By Lemma~\ref{vanishing of cohomology for family} and 
the same construction in Lemma~\ref{G-acyclicity} 
we have a one parameter family of $K$-equivariant $S^1$-acyclic system on $Y''$. 
Then the propsition follows from the deformation invariance of the index and the exicition prpoerty. 
\end{proof}

\subsection{Key proposition}
Recall the data $K_1$, $R$ and $(Y_1, L_1)$ 
considered in Subsection~\ref{product formula}. 
In the next subsection we will show the following proposition 
which is a special case of Theorem~\ref{main S^1-vanishing}. 
\begin{proposition}\label{special case}
$$
{\rm ind}_{K_1}^{S^1}(Y_1,W_1)=0. 
$$
\end{proposition}
Theorem~\ref{main S^1-vanishing} can be proved 
by using Proposition~\ref{special case} as follows. 

\begin{proof}[Proof of Thereom~\ref{main S^1-vanishing} 
assuming Proposition~\ref{special case}]
By taking $Y_1=R_{\nu}$ in Proposition~\ref{special case}
we have ${\rm ind}_K^{S^1}(M',W|_{M'})=0$ by 
Proposition~\ref{product formula for fixed point set 2}. 
\end{proof}

\subsection{Proof of the key proposition:one-dimensional case}
\label{one-dimensional case}
In this subsection we give the proof of Proposition~\ref{special case} 
in the case of $\dim Y_1=1$. 
All the following construction can be carried out $K_1$-equivariantly. 
We fix an isomorpshim $z:Y_1\to \C$ as Hermitian vector spaces. 
Define 
$
w:Y_1\setminus \{0\}\to \C/2\pi\sqrt{-1}\Z, \ 
\sigma : Y_1\setminus \{0\}\to\R \ {\rm and} \ 
\theta : Y_1\setminus \{0\} \to \R/2\pi\Z 
$ by $w=\log z=\sigma+\sqrt{-1}\theta$. 
Let $Y'_1$ be a manifold $Y_1$ whose complex structure is given by 
$z$ and K\"ahler metric is given by $|dz|=\sqrt{2}$ on $\sigma<-1$ and 
$|dw|=\sqrt{2}$ on $\sigma>1$. 
Let $D'$ be the Dolbeault operator on 
$W'_1=\wedge^{0,\bullet}T^*Y_1'\otimes R_1$. 
We consider an $S^1$-acyclic comptible sytem on $V:=\{\sigma>0\}$ 
defined by $S^1$-orbits and the de Rham operator $D'_V$ along fibers. 
Note that as in Remark~\ref{S^1-case} we use the open covering consisting of 
the single open set $V$ to define the structure of a compatible fibration.  
Fix a positive number $\rho_{\infty}$. 
Let $\rho:\R\to [0,\infty)$ be a non-negative smooth function 
such that $\rho(t)=0$ for $t<0$ and $\rho(t)=\rho_{\infty}$ for $t\gg 0$. 
For $t\ge 0$ we define a self adjoint elliptic operator $D_t'$ by 
$$
D_t':=D'+t\rho D_{V}':\Gamma(W')\to \Gamma(W'). 
$$
\begin{lemma}\label{dim Y_1=1}
For all $t\ge 0$ a section $s\in \Gamma(W')$ satisfies $D_t's=0$ and 
$\lim_{\sigma\to \infty}s=0$ if and only of $s=0$. 
\end{lemma}

\begin{proof}
We first give the proof for $t=0$. 
If $s\in \Gamma(W')$ is a degree zero section such that $D's=0$ 
and $\lim_{\sigma\to \infty}s=0$,  
then $s$ is a bounded holomorphic section, and hence $s=0$. 
Note that if a degree one section $s$ of $W$ satisfies $D's=0$, then  
$s$ has a form $s=f(\bar z)d\bar z$ for some anti-holomorphic function $f$. 
We show that if $s=f(\bar z)d\bar z$ satisfies 
$\lim_{\sigma\to \infty}s=0$, then $s=0$. 
Since we have $s=f(\bar z)d\bar z=f(\bar z)\bar z d\bar w$ and 
$s$ converges to zero at infinity, 
$f(\bar z)\bar z$ converges to zero at infinity. 
We put $u:=e^{-w}=1/z$ and $h(u):=f(z)z$. 
Then $h(u)$ has zero of order at least one at $u=0$, and hence, 
$$
s=f(\bar z)d\bar z=\frac{h(\bar u)}{\bar u}d\bar u
$$ is an anti-holomorphic 1-form on the one point compactification ${\mathbb P}^1$ 
of $Y'$ by Riemann's removable singularity theorem. 
Then we have $s=0$. 

We reduce the case $t>0$ to the case $t=0$. 
Let $\hat\sigma$ be a  smooth increasing function on $\sigma$ such that 
$\hat\sigma(\sigma)=\sigma$ for $\sigma<0$ and 
$$
\frac{d\hat\sigma}{d\sigma}=1+t\rho(\sigma).
$$
Let $\hat z:Y_1\to \C$ be a coordinate function defined by 
$$
\hat z=e^{\hat w}, \ \hat w=\hat \sigma+\sqrt{-1}\theta. 
$$
Let $Y''$ be a manifold $Y_1$ whose complex structure is given by 
$\hat z$ and K\"ahler metric is given by $|d\hat z|=\sqrt{2}$ 
on $\sigma<-1$ and $|d\hat w|=\sqrt{2}$ on $\sigma>1$. 
Let $D''$ be the Dolbeault operator on 
$W''_1=\wedge^{0,\bullet}T^*Y''\otimes R_1$. 
Since K\"ahler structures of $Y'$ and $Y''$ can be identified on $\{\sigma<0\}$,
$W'$ and $W''$ can be also identified on $\{\sigma<0\}$. 
We extend this isomorphism as $\phi : W'\to W''$ by defining 
$1\mapsto 1$ on the degree zero part and 
$\bar w\mapsto \overline{\hat w}$ on the degree one part. 
By the direct computation one can check that 
$$
(1+t\rho)D''=\phi\circ D_t'\circ \phi^{-1}. 
$$
Then $D_t's=0$ and $\lim_{\sigma\to\infty}s=0$ 
if and only if $D''(\phi(s))=0$ and $\lim_{\sigma\to \infty}\phi(s)=0$.
From the above argument for $t=0$ we have $\phi(s)=0$ and 
hence $s=0$. 
\end{proof}

Using Lemma~\ref{dim Y_1=1} Proposition~\ref{special case} 
can be proved as follows. 
\begin{proof}[Proof of Proposition~\ref{special case}]
By the deformation invariance of local indices we have 
${\rm ind}(Y_1, W_1)={\rm ind}(Y'_1,W'_1)$. 
Note that ${\rm ind}(Y'_1,W'_1)$ is defined as the super-dimension of 
the space of $L^2$-solutions of the equation $D_t's=0$. 
On the other hand any $L^2$-solution of $D_t's=0$ 
satisfies the boundary condition $\lim_{\sigma\to \infty}s=0$, and hence, 
we have ${\rm ind}(Y'_1,W'_1)=0$ by Lemma~\ref{dim Y_1=1}. 
In particular we have ${\rm ind}_K^{S^1}(Y_1, W_1)=0$. 
\end{proof}
\subsection{Proof of the key proposition:general case}
In this subsection we give the proof of Proposition~\ref{special case} 
for general $Y_1$. Since we assume $Y_1^{S^1}=\{0\}$ 
we have the decomposition of $Y_1$ into 
the positive weight part $E_+$ and the negative weight part $E_-$ 
as a representation of $S^1$, 
$$
Y_1=E_+\oplus E_-.
$$
Define a unitarty representation $E$ of $S^1$ by 
$$
E:=E_+\oplus \C\oplus E_-, 
$$where we consider the trivial $S^1$-action on $\C$. 
Now we define a $\C^*$-action on $E$ as follows: 
$\C^*$ acts by the complexification of the $S^1$-action on $E_+$ , 
the complexifiction of the $S^1$-action defined by $g\mapsto g^{-1}$ on $E_-$
and the standard multiplication of weight $1$ on $\C$. 
Note that all weights of the $\C^*$-action on $E$ are positive and 
the $\C^*$-action commutes with the $S^1$-action. 
The quotient space 
$$
Z:=(E\setminus \{0\})/\C^*
$$is a weighted projective space with the induced $K_1\times S^1$-action. 
\begin{lemma}\label{weighted vanishing}
\begin{enumerate}
\item
We have 
$$
H^{0,i}(Z,\CO_Z)=0 \qquad (i >0)
$$for the structure sheaf ${\mathcal O}_Z$ of $Z$. 
\item
Let $R$ be a non-trivial $1$-dimensional representation of $K_1\times S^1$. 
Then for $W_Z=\wedge^{0,\bullet} T^* Z \otimes R$ we have 
$$
{\rm ind}_{K_1}^{S^1}(Z,W_Z)=0 \in R(K_1). 
$$
\end{enumerate}
\end{lemma}
\begin{proof} 
(1)  
The first statement follows from the negativity of 
the canonical bundle $\det T^*Z$ of $Z$ and the Bochner trick.
We can show the negativity by the following (standard) argument.
Let $L_E$ be the trivial line bundle $E\times\C$ with 
a $\C^*$-action and the product holomorphic structure. 
Here 
$\C^*$ acts by weight one on the fiber $\C$. 
Let $S^1$ be the unit circle in $\C^*$. 
We define an $S^1$-invariant Hermitian structure on $L_E$ by 
$|\underline{v}|^2:=e^{-\pi r^2}|v|^2$ for a constant section with value 
$v\in \C$, where $r:E\to \R$ is the distance from the origin of $E$. 
Combining with the holomorphic structure and this Hermitian structure 
of $L_E$  we have the $S^1$-invariant Hermitian connection 
$\nabla_E$ on $L_E$ whose $(0,1)$-part coincides with 
the Dolbeault operator of $L_E$ and the curvature form coincides with 
the symplectic form on $E$. 
Since the $S^1$-action on $E$ lifts to $(L_E,\nabla)$, 
the $S^1$-action is Hamiltonian. 
Let $\mu:E\to \sqrt{-1}\R$ be its moment map. 
Note that $\mu$ is proper because all weights of the $S^1$-action are positive. In this setting the complex quotient $Z$ can be identified with 
the symplectic quotient $\mu^{-1}(0)/S^1$. 
In particular $Z$ is a K\"ahler manifold with 
the induced prequantizing line bundle $(L_Z,\nabla_Z)$, and we have 
an isomorphic as holomorphic Hermitian line bundles 
$$
\det TZ\cong (E\setminus\{0\})\times_{\C^*}\det E\cong L_Z^d, 
$$where $d>0$ is the weight of $S^1$-action on $\det E$. 
Since $L_Z$ is positive, the canonical bundle $\det T^*Z$ is negative. 

(2) 
Let $R$ be a non-trivial $1$-dimensional representation of $K_1\times S^1$. 
Since the K\"ahler structure of $Z$ is $K_1\times S^1$-equivariant we have 
a $K_1\times S^1$-equivariant isomorphism 
$$
H^{0,i}(Z,\CO_Z\otimes R)\cong H^{0,i}(Z,\CO_Z)\otimes R
$$for $i\ge 0$. By the first statement and the Hodge theory we have 
$$
{\rm ind}_{K_1\times S^1}(Z,W_Z)=\sum_{i}(-1)^iH^{0,i}(Z,\CO_Z\otimes R)=R
\in R(K_1\times S^1).
$$
Since $R$ is a non-trivial representaion of $S^1$ we have 
${\rm ind}^{S^1}_{K_1}(Z,W_Z)=0$. 
\end{proof}

Note that we have a decomposition of the fixed point set $Z^{S^1}$ as 
$$
Z^{S^1}=Z_-\sqcup Z_0 \sqcup Z_+, 
$$where we put 
$Z_-:=(E_+\setminus\{0\})/\C^*$, $Z_0:=(\C\setminus\{0\})/\C^*$ 
and $Z_+:=(E_-\setminus\{0\})/\C^*$. 
Then by the localization formula Theorem~\ref{equivariant localization for invariant part}, ${\rm ind}_{K_1}^{S^1}(Z,W_Z)$  
is equal to the sum of the contribution from $Z_-$, $Z_0$ and $Z_+$: 
$$
{\rm ind}_{K_1}^{S^1}(Z,W_Z)={\rm ind}_-+{\rm ind}_0+{\rm ind}_+. 
$$
By definition the contribution ${\rm ind}_0$ from $Z_0$ 
is equal to ${\rm ind}_{K_1}^{S^1}(Y_1,W_1)$ and hence we have 
$$
{\rm ind}_{K_1}^{S^1}(Y_1,W_1)=-({\rm ind}_-+{\rm ind}_+)
$$from Lemma~\ref{weighted vanishing}. 
Then we can show Proposition~\ref{special case} as follows. 
\begin{proof}[Proof of Proposition~\ref{special case}]
It is enough to show ${\rm ind}_-={\rm ind}_+=0$. 

(0) 
If the complex codimension of $M^{S^1}$ in $M$ is $1$, then 
we have ${\rm ind}_K^{S^1}(M',M'\setminus M^{S^1})=0$ 
by the product formula (Proposition~\ref{product formula for fixed point set}) 
and Proposition~\ref{dim Y_1=1}.

(1) If $E_-=0$, then we have ${\rm ind}_-=0$ and 
the codimension of $Z_+$ is one.  
In this case we have ${\rm ind}_+=0$ from the above case (0). 

(2) If $E_+=0$, we have ${\rm ind}_-={\rm ind}_+=0$ 
as in the same way for (1). 

(3) In the  general case note that all the weights of the $S^1$-action on the 
normal direction on $Z_-$ are positive.  
By (1) and the product formula we have ${\rm ind}_-=0$. 
As in the same way we have ${\rm ind}_+=0$ by (2) and the product formula. 
\end{proof}

\section{Localization of equivariant Riemann-Roch numbers}
Let $(M,\omega)$ be a closed prequantized symplectic manifold 
with $G$-equivariant prequantizing line bundle $(L,\nabla)$, 
i.e., $L$ is a Hermitian line bundle over $M$ with a Hermitian connection $\nabla$ whose curvature form is equal to $-2\pi\sqrt{-1}\omega$, and all these data are $G$-equivariant.
Let $G$ be a torus and $K$ a compact Lie group. 
Suppose that $G\times K$ acts effectively on $(M,\omega)$ 
and the $G\times K$-action on $M$ lifts to $L$ 
preserving  $\nabla$ and the Hermitian metric of $L$. 
Then, the action is Hamiltonian and each $G$-orbit is an isotropic torus in $M$. We denote the moment map for $G$-action associated with the lift by $\mu_G:M\to \g^*$. 
We assume that $0\in \g^*$ is a regular value of $\mu$. 

For these data one can define the $G\times K$ 
equivariant Riemann-Roch number $RR_{G\times K}(M,L)$ as 
the index of the Dolbeault operator whose coefficient is in $L$. 
We denote its $G$-invariant part by $RR_K^G(M,L)$. 
Note that $RR_K^G(M,L)$ is an element of the character ring $R(K)$ of $K$.

On the other hand we have the quotient space 
$M_G=\mu^{-1}_G(0)/G$, the symplectic reduction of $M$ at $0$.  
Since we assume that $0$ is a regular value 
$M_G$ is a closed symplectic orbifold and 
has the natural induced $K$-action and the $K$-equivariant 
prequantizing line bundle $(L_G,\nabla_G)=(L|_{\mu^{-1}_G(0)},\nabla|_{\mu^{-1}_G(0)})/G$.  
In this section we show the following theorem 
by induction on dimension of $G$. 

\begin{theorem}\label{equiv[Q,R]=0}
$$
RR_{K}^G(M,L)=RR_K(M_G,L_G)\in R(K). 
$$
\end{theorem}

As a special case we have a proof of Guillemin-Sternberg conjecture 
for the torus action.

\begin{theorem}
[\cite{Duistermaat-Guillemin-Meinrenken-Wu,Guillemin-Sternberg,Guillemin,Meinrenken1,Meinrenken2,Tian-Zhang,Vergne1,Vergne2}, etc.]
$$
RR^G(M,L)=RR(M_G,L_G)\in \Z. 
$$
\end{theorem}
\begin{remark}
The Guillemin-Sternberg conjecture itself is valid not only for a torus but also for a compact Lie group. 
\end{remark}

\subsection{Acyclic compatible system and local Riemann-Roch number}
\begin{definition}[$L$-acyclic point and $(L,G)$-acyclic point]
A point $\xi\in \mu_G(M)$ is called {\it $L$-acyclic} if 
the restriction of $(L,\nabla)$ to each orbit in $\mu_G^{-1}(\xi)$ 
does not have any non-trivial parallel sections. 
If the restriction does not have any 
non-trivial $G$-invariant parallel sections, then we call $\xi$ a 
{\it $(L,G)$-acyclic point}. 

\end{definition}

\begin{remark}
Since each $G$-orbit is an isotropic torus, by 
\cite[Lemma~5.12]{Fujita-Furuta-Yoshida2} and the Hodge theory, a point $\eta\in \mu_G(M)$ is $L$-acyclic if and only if, for each orbit ${\mathcal O}$ which is contained in $\mu^{-1}_G(\eta)$, the de Rham operator on ${\mathcal O}$ with coefficients in $L|_{\mathcal O}$ has zero kernel. 
As in the similar way, 
$\eta\in \mu_G(M)$ is $(L,G)$-acyclic if and only if 
the $G$-invariant part of the de Rham operator on the orbit is trivial.  

\end{remark}
Let $\g_{\Z}$ be the integral lattice and $\g^*_\Z$ the weight lattice of $G$. 
\begin{proposition}\label{non-L-acyclic=integral}
Non $L$-acyclic points are contained in $\g^*_\Z$. 
\end{proposition}
\begin{proof}
First let us recall that the moment map associated to the lift of the $G$-action on $(M,\omega)$ to $(L,\nabla)$ is defined by the following equality
\begin{equation}\label{char of mu_G}
\dfrac{d}{dt}\Big|_{t=0}\psi_{\exp -t\xi}\circ s\circ \varphi_{\exp t\xi}=\nabla_{X_\xi}s-2\pi \sqrt{-1}\< \mu_G ,\xi\>s
\end{equation}
for $\xi \in \g$ and $s\in \Gamma (L)$, where $\varphi_g$ denotes the $G$-action on $M$ and $\psi_g$ denotes the lift of $\varphi_g$ to $L$ for $g\in G$. 

Let $\eta^* \in \g^*$ be a non $L$-acyclic point and ${\mathcal O}$ a non $L$-acyclic orbit that is contained in $\mu_G^{-1}(\eta^*)$. Then, there exists a non-trivial parallel section which we denote by $s\in \Gamma (L|_{\mathcal O})$. For arbitrary element $\xi \in \g_{\Z}$ and $x\in {\mathcal O}$, by \eqref{char of mu_G} we have
\[
\dfrac{d}{dt}\psi_{\exp -t\xi}\circ s\circ \varphi_{\exp t\xi}(x)=-2\pi \sqrt{-1}\< \eta^* ,\xi\>\psi_{\exp -t\xi}\circ s\circ \varphi_{\exp t\xi}(x). 
\]
This implies that 
\begin{equation*}
\psi_{\exp -t\xi}\circ s\circ \varphi_{\exp t\xi}(x)=e^{-2\pi\sqrt{-1}t\< \eta^*,\xi\>}s(x) .
\end{equation*}
Since $\xi \in \g_{\Z}$, by putting $t=1$, we have 
\[
s(x)=\psi_{\exp -\xi}\circ s\circ \varphi_{\exp \xi}(x)=e^{-2\pi\sqrt{-1}\< \eta^*,\xi\>}s(x) .
\]
Then, $\< \eta^*,\xi\>$ must be in $\Z$. 
\end{proof}


\begin{remark}
There exist the following three conditions which are related to acyclicity on 
a $G$-invariant open subset $V'$ of $M$. 
\begin{itemize}
\item[(a)] Every point in $\mu_G(V')$ is $L$-acyclic.
\item[(b)] $\mu_G(V')$ does not contain any integral points. 
\item[(c)] The compatible system on $V'$ is acyclic.
\end{itemize}
Neither the condition (a) nor (c) implies (b) because 
there may exist a lattice point
without non-zero parallel section on its fiber. 
The condition (b) implies (a) by Proposition~\ref{non-L-acyclic=integral}. 
The condition (c) implies (a) because if every $H^{\perp}$-orbit does not 
have any non-zero parallel sections, then every $G$-orbit does not have. 
Conversely neither the condition (a) nor (b) implies (c).  
%
When $G=S^1$, the conditions (a) and (c) are equivalent because the index set $A$ of the compatible fibration consists of the single element $\{ e\}$.  
In this case, since the action of $G$ on $H^*(G,\C)$ is trivial,
two conditions (a') $(L,G)$-acyclicity and (c') $G$-acyclicity are also equivalent. 
Note that only the condition (c) (resp. (c')) 
is our sufficient condition to define the index (resp. $G$-invariant index), 
and we only use it.   
\end{remark}

Let $V$ be a $G\times K$-invariant open subset of $M$. 
We fix a $G\times K$-invariant $\omega$-compatible 
almost complex structure on $V$. 
As in Subsection~\ref{torus : compatible fibration} 
$V$ has a structure of $G\times K$-equivariant 
compatible fibration with the open covering $\{V_H\}_{H\in A}$ 
parameterized by the set of subgroups of $G$ which appear as 
the identity components of the stabilizers of the $G$-action. 
Consider the $\Z/2$-graded Clifford module bundle 
$W_L=\wedge^{0,\bullet}T^*V\otimes L|_V$.  
Suppose that there is a $G\times K$-invariant open subset $V'$ of $V$ 
such that the family of Dirac type operators along leaves $\{D_H\}$ 
constructed as in Subsection~\ref{torus : compatible system} 
defines a $G\times K$-equivariant 
acyclic (resp. $G$-acyclic) compatible system. 
Using these data we have the equivariant local index 
${\rm ind}_{G\times K}(V,V',W_L)\in R(G\times K)$ 
(resp. ${\rm ind}_{K}^G(V,V',W_L)\in R(K)$). 

\begin{definition}[Equivariant Riemann-Roch number]
For the above data we define the {\it equivariant Riemann-Roch number} 
$RR_{G\times K}(V,V',L)\in R(G\times K)$ by putting 
$RR_{G\times K}(V,V',L)={\rm ind}_{G\times K}(V,V',W_L)$. 
As in the similar way when $\{D_{\alpha}\}$ is $G$-acyclic 
the {\it $G$-invariant part of the Riemann-Roch number} 
$RR_K^G(V,V',L)\in R(K)$ is defined by
$RR_{K}^G(V,V',L)={\rm ind}_{K}^G(V,V',W_L)$. 
\end{definition}

\subsection{Proof of Theorem~\ref{equiv[Q,R]=0}: $S^1$-case}
\label{Proof of S^1-case}
We first show Theorem~\ref{equiv[Q,R]=0} in the case of $\dim G=1$. 
%
For each lattice point $k\in \g_{\Z}^*$ we can take a small 
$G\times K$-invariant open neighborhood $V_k$ 
of the compact set $\mu_G^{-1}(k)$
so that the image of the complement $\mu_G(V_k\setminus \mu^{-1}_G(k))$ 
consists of $L$-acyclic points. 
The proof of Proposition~\ref{non-L-acyclic=integral} implies that 
if $k$ is not equal to $0$, then the image of the complement of the fixed point set 
$\mu_G(V_k\setminus \mu^{-1}_G(k)^{G})$ consists of 
$(L,G)$-acyclic points. 

\begin{lemma}\label{localization}
$$
RR_{K}^{G}(M,L)=
RR_{K}^{G}(V_0,V_0\setminus\mu^{-1}_G(0),L|_{V_0})\in R(K). 
$$
\end{lemma}
\begin{proof}
By Proposition~\ref{non-L-acyclic=integral} 
the complement $M\setminus \mu_G^{-1}(\g_{\Z}^*)$ has a structure of 
$G\times K$-equivariant acyclic compatible system and 
$RR_{G\times K}(M,L)$ is equal to the sum of 
equivariant local Riemann-Roch numbers 
$RR_{G\times K}(V_k, V_k\setminus \mu^{-1}_G(k), L|_{V_k})$. 
Moreover if $k\neq 0$, then $V_k\setminus \mu^{-1}_G(k)^G$ has  
a structure of $K$-equivariant $G$-acyclic compatible system 
we have the equality for the $G$-invariant part, 
$$
RR_K^{G}(M,L)=RR_K^{G}(V_0,V_0\setminus\mu^{-1}_G(0), L|_{V_0})
+\sum_{k\in \g_{\Z}^*\setminus\{0\}}
RR_K^{G}(V_k,V_k\setminus\mu_G^{-1}(k)^{G}, L|_{V_k}).
$$
On the other hand we have 
$$
RR_K^{G}(V_k,V_k\setminus\mu_G^{-1}(k)^{G},L|_{V_k})=0 
$$for $k\neq 0$ by Theorem~\ref{main S^1-vanishing}, 
and hence, we complete the proof.  
\end{proof}

Now we identify the neighborhood $V_0$ of $\mu^{-1}_G(0)$. 
Let $T^*G$ be the cotangent bundle of $G$ with 
the prequantizing line bundle $L(0)$ as in Subsection~\ref{cotangent bundle case}, where $0\in \g_{\Z}^*$ corresponds to 
the one dimensional trivial representation of $G$.

\begin{lemma}\label{local model}
There is a $G\times K$-equivariant symplectomorphism between 
small open neighborhoods of $\mu^{-1}_G(0)$ in $M$ and 
the zero section in $\mu^{-1}_G(0)\times_{G}T^*G(\cong 
\mu^{-1}_G(0)\times\g^*)$. 
Moreover this symplectomorphism lifts to 
a $G\times K$-equivariant isomorphism between 
prequantizing line bundles over them. 
\end{lemma}

This lemma follows from the following general proposition, 
which is a generalization of Darboux's theorem 
(\cite{McDuffSalamon}). 

\begin{proposition}
Let $H$ be a compact Lie group 
acting on symplectic manifolds $(M_i,\omega_i)$ for $i=0,1$. 
We assume each $(M_i,\omega_i)$ has a $H$-equivariant 
prequantizing line bundle $(L_i,\nabla_i)$ . 
Suppose that there is a compact manifold $N$ 
with the following properties. 
\begin{enumerate}
\item There is an embedding $\iota_i:N\hookrightarrow M_i$. 
\item Normal bundles $\nu_{\iota_i(N)}$ of $\iota_i(N)$ in $M_i$ are 
$H$-equivariantly isomorphic to each other. 
\item The collections of data $\iota_i^*(\omega_i,L_i,\nabla_i)$ 
are $H$-equivariantly isomorphic. 
\end{enumerate}
Then there is a $H$-invariant neighborhood $U_i$ of $\iota_i(N)$ and 
$H$-equivariant diffeomorphism $\phi:U_1\to U_2$ such that 
\begin{enumerate}
\item The following diagram commutes. 
\[
\xymatrix{
U_1 \ar[rr]^{\phi} && U_2  \\ 
& N \ar[lu]^{\iota_1} \ar[ur]_{\iota_2} &
}\]
\item The diffeomorphism $\phi$ lifts to a $H$-equivariant isomorphisms 
between $(\omega_i,L_i,\nabla_i)$. 
\end{enumerate}
\end{proposition}

\begin{proof}
For $i=0,1$ consider the associated principal $U(1)$-bundle with connection 
1-form $(P_i,\alpha_i)$ of $(L_i,\nabla_i)$. 
Then the proposition follows from $G\times U(1)$-equivariant 
Morser's argument for contact manifold $P_i$ with contact 1-form $\alpha_i$ for 
submanifold $\iota_i(N)$. 
Here the $U(1)$-action comes from that of the 
structure of a principal bundle. 
\end{proof}

Note that $\mu^{-1}_G(0)\to \mu^{-1}_G(0)/G=M_G$ has a structure of a 
principal $G$-bundle (in the sense of orbifold), 
and $\mu^{-1}_G(0)\times_G T^*G$ has the structure of a product of 
compatible fibrations in the sense of Subsection~\ref{product formula}. 
Here the compatible fibration on $M_G$ is the trivial one 
(all leaves are a point) 
and the one on $T^*G$ is induced by the induced $G$-action. 
Since the symplectomorphism in Lemma~\ref{local model} is $G$-equivariant 
and preserves the prequantizing line bundles, we have the following. 

\begin{lemma}\label{local model-acyclic system}
The symplectomorphism in Lemma~\ref{local model} 
induces an isomorphism between two acyclic compatible systems. 
\end{lemma}

We take $V_0$ to be the open 
neighborhood as in Lemma~\ref{local model}. 
Then by Proposition~\ref{index of T^*S^1} and 
the product formula (Proposition~\ref{product formula for fixed point set}), we have 
$$
RR_K^G(V_0,V_0\setminus \mu_G^{-1}(0),L|_{V_0})=
RR_K(M_G, L_G) 
$$and hence by Lemma~\ref{localization} we complete the proof of Theorem~\ref{equiv[Q,R]=0} in the case of $\dim G=1$. 

\begin{remark}
The argument in this subsection 
implies that Theorem~\ref{equiv[Q,R]=0} 
holds in the case that $M$ is not necessarily closed. 
Let $(M,\omega)$ be a prequantized symplectic manifold 
with prequantizing line bundle $(L,\nabla)$. 
We do not assume that $M$ is closed. 
Let $G$ be the circle group $S^1$ and $K$ a compact Lie group. 
Suppose that $G\times K$ acts effectively on $(M,\omega)$ 
and the $G\times K$-action on $M$ lifts to $L$ 
preserving  $\nabla$ and the Hermitian metric of $L$. 
We assume that the corresponding moment map $\mu_G$ of the $G$-action 
is a proper map and $0$ is its regular value. 
Suppose that there is a $G$-invariant open subset $V$ of $M$ such that 
the complement $M\setminus V$ is a compact neighborhood of $\mu_G^{-1}(0)$ 
and the image $\mu_G(M\setminus V)$ 
does not contain any non-zero integral point. 
Under the above assumptions we have the equivariant index 
$RR_{G\times K}(M,V,L)$. 
As in the same way in the proof of Lemma~\ref{localization} 
and by Lemma~\ref{local model-acyclic system}, 
we have  
$$
RR_K^G(M,V,L)=RR_K(M_G,L_G), 
$$where $M_G$ is the symplectic reduction of $M$ at 0 and 
$L_G$ is the induced prequantizing line bundle on $M_G$. 
Note that since $\mu_G$ is proper $M_G$ is a closed symplectic orbifold. 


On the other hand, when $M$ is not necessarily compact, Vergne 
\cite{Vergne3} formulated a 
version of quantization conjecture for general compact Lie group $G$
in terms of transversally elliptic operator \cite{Atiyah1}. 
The conjecture was proved by Ma and Zhang \cite{Ma-Zhang}, and  
Paradan \cite{Paradan2} gave a new proof of it. 
In \cite{Hajime} the author gave a formulation of a variant of $S^1$-equivariant index for non-compact symplectic manifold in the context of this paper. 
The $S^1$-equivariant index coincides with the transverse index for proper moment map without critical points. 
If the moment map is not proper, then, they however do not coincide in general. 
See \cite{Hajime} for details. 
\end{remark}

\subsection{Proof of Theorem~\ref{equiv[Q,R]=0}: general case}
To show Theorem~\ref{equiv[Q,R]=0} by induction on $\dim G$ 
we have to choose an appropriate circle subgroup as follows. 
\begin{lemma}\label{generic S^1}
There exists a compact connected 1 dimensional subgroup 
$H$ of $G$ such that the induced moment map 
$\mu_H=\iota_H^*\circ\mu_G:M\to \h^*$ has $0$ as its regular value, 
where ${\mathfrak h}$ is the Lie algebra of $H$ and 
$\iota_H:{\mathfrak h}\to \g$ is the natural inclusion. 
\end{lemma}
\begin{proof}
Let $A$ be the set of all subgroups of $G$ 
which appear as a stabilizer of $G$-action on $M$. 
Since $M$ is compact 
$A$ is a finite set 
and the image of the fixed point set $\mu_G(M^G)$ is a finite set 
which does not contain $0$, and  
we can choose a one-dimensional subspace ${\mathfrak h}$ in 
${\mathfrak g}$ which satisfies the following conditions. 
\begin{enumerate}
\item ${\mathfrak h}$ is generated by rational vectors. 
\item ${\mathfrak h}\setminus\{0\}$ does not intersect with 
$\displaystyle\bigcup_{G'\in A\setminus \{G\}}\text{Lie}~G'$ 
in $\g$. 
\item $\mu_G(M^G)$ does not intersect with $\ker\iota^*_H$. 
\end{enumerate}
We define $H$ as the one-dimensional connected subgroup of $G$ which 
is defined as the image of ${\mathfrak h}$ by the exponential map. 
Note that by the above first condition $H$ is a compact subgroup of $G$,  
and we have $M^H=M^G$ from the second condition. 
Since $H$ is a one-dimensional subgroup 
$x\in M$ is a critical point of $\mu_H=\iota^*_H\circ\mu_G$ 
if and only if it is a fixed point $x\in M^H=M^G$.
By the last condition for ${\mathfrak h}$, 
we have that $0$ is a regular value of $\mu_H$. 
\end{proof}

\begin{proof}[Proof of Theorem~\ref{equiv[Q,R]=0}]
We show the theorem by induction on $\dim G$. 
As in the previous subsection we proved in the case of $\dim G=1$. 
We take a one-dimensional connected subgroup $H$ of $G$ 
as in Lemma~\ref{generic S^1} and a complementary subtorus $G'$. 
According to this decomposition we have a decomposition of the moment map 
$\mu_G$ as $\mu_G=\mu_H\oplus\mu_{G'}:M\to \g^*={\mathfrak h}^*\oplus(\g')^*$. 
Let $\bar\mu_{G'}:\mu_H^{-1}(0)/H\to({\mathfrak g}')^* $ 
be the induced moment map with respect to the induced $G'$-action. 
Since the natural projection 
$\mu^{-1}_G(0)=(\mu_{G'}|_{\mu^{-1}_H(0)})^{-1}(0)\to\bar\mu_{G'}^{-1}(0)$ 
is a submersion and $0$ is a regular value of $\mu_G$ and $\mu_H$, 
$0$ is also a regular value of $\bar\mu_{G'}$. 
Then we can prove Theorem~\ref{equiv[Q,R]=0} inductively as follows: 
\begin{eqnarray*}
RR_K^G(M,L)&=&(RR_{K}^{H}(M,L))^{G'} \\
&=&(RR_K(M_H,L_H))^{G'} \\ 
&=&RR_K((M_H)_{G'}, ({L_H})_{G'}) \\ 
&=&RR_K(M_G,L_G).   
\end{eqnarray*} 
Here the second equality follows from the facts
that $H$ is one-dimensional and $0$ is 
a regular value of $\mu_H$, 
and the third equality follows from the facts that 
$G'$ is $m-1$-dimensional and $0$ is a regular value of $\bar\mu_{G'}$. 
\end{proof}

\noindent{\bf Acknowledgements.}
The first author is partly supported 
by Grant-in-Aid for Young Scientists (Start-up) 21840045 and Grant-in-Aid for Young Scientists (B) 23740059.
The second author is
partly supported by Grant-in-Aid for Scientific 
Research (A) 19204003 and Grant-in-Aid for Scientific Research (B) 19340015.
The third author is partly supported by Grant-in-Aid 
for Young Scientists (B) 22740046, and Fujyukai Foundation.
\bibliographystyle{amsplain}
\bibliography{reference}
\end{document}